\documentclass[a4paper,10pt]{article}
\usepackage{amssymb}
\usepackage{amsmath}
\usepackage{graphicx}

\DeclareGraphicsExtensions{.pdf,.png,.jpg}

\newcommand{\eps}{\varepsilon}
\newcommand{\A}{\mathcal A}
\newcommand{\Q}{\mathcal Q}
\newcommand{\Oq}{\overline{\mathcal Q}}
\newcommand{\R}{\mathcal R}

\newcommand{\im}{\mbox{\rm Im}\,}
\newcommand{\re}{\mbox{\rm Re}\,}

\def\qed{{\hfill{\vrule height5pt width3pt depth0pt}\medskip}}

\newcommand{\proof}{{\bf Proof.\ \ }}

\newtheorem{theorem}{Theorem}
\newtheorem{lemma}[theorem]{Lemma}
\newtheorem{cor}{Corollary}
\newtheorem{definition}{Definition}
\newtheorem{remark}[theorem]{Remark}

\def\CC{\mathbb{C}}

\def\NN{\mathbb{N}}
\def\DD{\mathbb{D}}

\def\OO{{\mathcal O}}

\def\eps{\varepsilon}

\DeclareMathOperator{\arctanh}{arctanh}

\renewcommand{\phi}{\varphi}
\title{New lower bound estimates for quadratures  of bounded analytic functions}
\author{Ma{\l}gorzata Moczurad, Piotr Zgliczy\'nski  and W{\l}odzimierz Zwonek}

\begin{document}

\maketitle
\begin{center}
 Jagiellonian University, \\
 Faculty of Mathematics and Informatics, \\
{\L}ojasiewicza 6, 30--348  Krak\'ow, Poland \\

 \vskip 0.5cm
e-mail: malgorzata.moczurad@ii.uj.edu.pl,
piotr.zgliczynski@ii.uj.edu.pl, wlodzimierz.zwonek@im.uj.edu.pl

\vskip 0.5cm
\end{center}

\begin{abstract}
We give an improved lower bound for the error of any quadrature formula for calculating $\int_{-1}^1 f(x) d\alpha(x)$, where the functions $f$ are bounded and analytic in the neighborhood of [-1,1] and  $\alpha$  is finite absolutely continuous Borel measure.
\end{abstract}


\section{Introduction}
\label{sec:intro}


To facilitate the discussion we introduce first the basic notation and definitions.

\subsection{Basic notation and definitions}

\begin{definition}
Let $c > 1$. By $\mathcal{E}_c$ we denote an interior of an ellipse, such that
the foci of $\mathcal{E}_c$ are located at points $\pm 1$ and the
sum of semi-axes is equal to $c$.
\end{definition}

\begin{definition}
Let $D \subset \mathbb{C}$. We call $D$ a {\em nice domain} if it is an open, connected and simply connected set, symmetric with respect to the real axis (i.e.\ if $z \in D$, then $\bar{z} \in D$).
\end{definition}

\begin{definition}
Let $D \subset \mathbb{C}$ be a nice domain.  Let $M \geq 0$.

We will write:
\begin{itemize}
\item[] $\A(D)$ for the set of analytic functions
on $D$ such that $\|f\|=\sup_{z \in D}|f(z)| < \infty$,

\item[] $\A(D,M)$ for the set of analytic functions
on $D$ such that $|f(z)| \leq M$ for $z \in D$,

\item[] $\A_0(D,M)$ for a subset of $\A(D,M)$ consisting of the functions
which are real on the real line.
\end{itemize}
\end{definition}

We denote by $I(f,\alpha)$ the integral
\begin{equation*}
  I(f,\alpha)=\int_{-1}^1 f(x) d\alpha(x),  \label{eq:int-with-weight}
\end{equation*}
where $\alpha$ is a finite Borel measure on $[-1,1]$ which is absolutely continuous with respect to the Lebesgue measure. Usually we  drop $\alpha$
and write $I(f)$, when $\alpha$ is known from the context.  Let $\Q(n,\R)$, where $n \in \mathbb{N}$ and $\R = (r_1, \ldots, r_n)$, $r_1, \ldots, r_n\in \NN\backslash \{0\}$, denote the class of all possible (even non-linear) quadratures that use $n$ nodes $z_1,\dots,z_n \in [-1,1]$ and derivatives of an integrand up to the order $r_j - 1$ for each $z_j$. By $\Oq(n, \R)$ we denote a subclass of $\Q(n,\R)$ containing quadratures of the form
\begin{equation}
 S_\mathcal{R}(f)=\sum_{j=1}^n  \sum_{k=0}^{r_j - 1}  b_{kj} f^{(k)}(z_j). \label{eq:quad-with-der}
\end{equation}
Additionally,  $|\R|$ denotes the sum $r_1 + \ldots + r_n$ and $\R_2 = (2, \ldots, 2)$.

Following \cite{P98} we introduce the following definitions.
\begin{definition}
\label{def:petras-optimal}
Let $D \subset \mathbb{C}$ be open set, such that $[-1,1] \subset D$.
For a given quadrature $Q \in\Q(n,\R)$ the remainder term is defined as
\begin{equation}
  R(f,\alpha)=I(f,\alpha) - Q(f).
\end{equation}
The error constant of $Q$ with respect to $\A(D)$ is given by
\begin{equation}
  \rho(Q,\A(D),\alpha)=\sup_{f \in \A(D) \setminus \{0\}} \frac{|R(f, \alpha)|}{\|f\|}
\end{equation}
and the respective optimal error constant by  $ \rho_n(\A(D),\alpha)$
\begin{equation}
  \rho_n(\A(D),\alpha)=\inf_{Q\in \Q(n,\R_2)} \rho(Q,\A(D),\alpha).
\end{equation}

  A quadrature formula is called \emph{optimal} if its error constant attains $\rho_n(\A(D),\alpha)$.
\end{definition}

Let us stress that in the above definition only $\R_2$ is used to define the optimal error constant and optimal quadrature .

To measure the quality of a  quadrature formula $Q_n \in \Oq(n,\R_2)$ Petras in \cite{P98} proposed the following definition.
\begin{definition}
\label{def:petras-loss}
  \begin{equation}
    \mbox{\rm loss}(Q,\A(D),d\alpha)=\frac{\rho(Q,\A(D),\alpha)}{\rho_n(\A(D),\alpha)}.
  \end{equation}
  The sequence $\{Q_n\}_{n \in \mathbb{N}}$, where $Q_n\in \Oq(n,\R_2)$, is called \emph{near-optimal}, if the sequence of corresponding losses is bounded.
\end{definition}


\subsection{Our motivation and main result}

In the works of Bakhvalov \cite{B} and Petras \cite{P98} there are  convincing arguments for  the near-optimality of the Gaussian quadrature in the case
when the domain of analyticity of the integrand is an ellipse; for other regions, it will be the Gaussian quadrature transported from the unique
ellipse via the Riemann mapping theorem. In Petras' article \cite{P98} one can find a demonstration of how the Gaussian quadrature fails to be nearly optimal,
 when the analyticity region is not an ellipse.

To describe briefly the results of Bakhvalov and Petras we  assume that $\alpha$ is the Lebesgue measure and $\R=\R_2$  (results in \cite{B} and \cite{P98} have been established for more general situations  discussed in more details in Section~\ref{sec:lower-upper-bnds}). Let $G_n$ denotes the Gauss-Legendre quadrature with $n$ nodes.

 The claim of an almost optimal performance of the Gauss-Legendre quadrature formula for ellipses is based on the following estimates
\begin{itemize}
\item there exists a bounded  and positive function $\kappa_l:(1,\infty) \to \mathbb{R}_+ $, such for any $c>1$ and for any quadrature $Q_n\in \Oq(n,\R_2)$  there is an  $f_0 \in  A_0(\mathcal{E}_c,M)$ such that
\begin{equation*}
|I(f_0) - Q_n(f_0)| \geq M \kappa_{l}(c) c^{-2n},  \label{eq:lo-bnd-err-intro}
\end{equation*}

\item there exists a bounded  and positive function $\kappa_g:(1,\infty) \to \mathbb{R}_+ $ such that, for any $c>1$ for the Gauss-Legendre quadrature $G_n$ for any  $f  \in  A_0(\mathcal{E}_c,M)$ holds
\begin{equation*}
|I(f) - G_n(f)| \leq M \kappa_{g}(c) c^{-2n}. \label{eq:up-bnd-err-intro}
\end{equation*}
\end{itemize}
Observe that the above estimates lead to  asymptotically the same bounds for $n$ needed to get the
quadrature error less than $\eps$.  We obtain
\begin{equation}
  N_l\left(\frac{M}{\eps},c \right)\leq n \leq N_g\left(\frac{M}{\eps},c \right)
\end{equation}
where
\begin{eqnarray}
   N_l\left(\frac{M}{\eps},c \right)=\frac{\ln \frac{M}{\eps} + \ln \kappa_l(c)}{2 \ln c}, \qquad N_g\left(\frac{M}{\eps},c \right)=\frac{\ln \frac{M}{\eps} + \ln \kappa_g(c)}{2 \ln c}. \label{eq:estm-Nl-Ng}
\end{eqnarray}
For $\eps \to 0^+$ we have $N_l \approx N_g \approx \left(\ln \frac{M}{\eps}\right)/(2 \ln c)$, so both lower and upper bounds predict more
or less the same number of nodes.

The motivation for our work comes from the following observation. From the estimates for $\kappa_l(c)$ given in \cite{B,P98} it follows that
\begin{equation}
  \lim_{c \to 1^+} \kappa_l(c)=0.  \label{eq:kappa-to-0}
\end{equation}
Thus if $c-1$ is small, $N_l<0$ in (\ref{eq:estm-Nl-Ng}) unless $\eps$ is very small, so in fact the lower bound given by (\ref{eq:estm-Nl-Ng}) does not have
any predictive power w.r.t. the number of nodes required to get the error less than $\eps$ for a substantial range of the parameters $c$ and $\eps$.


 The main technical result of our paper is  a new lower bound for  errors of arbitrary quadratures of bounded analytic function using $N$ values of functions or its  derivatives at some nodes,  which
does not suffer from the bad qualitative behavior exemplified by equation (\ref{eq:kappa-to-0}). This allows to obtain more meaningful lower
bounds on the cost of quadratures in the sprit of IBC approach to the complexity of integration
of bounded analytic functions (see \cite{KWW} and references given there).

Our approach is based on the conformal distance  on the domain of analyticity $D$.
The theorem below is an example of our  lower bound for the case of the Lebesgue measure.

\begin{theorem} \label{thm:main}
Let $D \subset \mathbb{C}$, $D \neq \mathbb{C}$ and let $D$ be a nice domain
such that $[-1,1] \subset D$.
For any $Q\in {\mathcal Q}(n, \R)$, where $|\R| = N$, and for any  $M>0$ there
exists a function $f_0 \in \A_0(D,M)$ such that
\begin{equation}
  |I(f_0) - Q(f_0)| \geq \gamma M,\label{eq:from-below}
\end{equation}
where
\begin{equation}\label{eq:gamma}
   \gamma = \left((1+1/(2\delta_D))^{2\delta_D}(2\delta_D+1)\right)^{-2N}
\end{equation}
and
\begin{equation*}
\delta_D:=\sup\{\delta_D(x):x\in[-1,1]\}, \quad \delta_D(x):=\inf\{|x-z|:z\in\mathbb{C}\setminus D\}.
\end{equation*}
\end{theorem}

This theorem is proved in Section~\ref{sec:new-lower-bound}.
Corollary \ref{cor:ellipse} therein contains the version of this result for the ellipse~$\mathcal{E}_c$.

 This result improves  the results of Bakhvalov \cite{B} and Petras \cite{P98} as it allows higher derivatives in the quadrature formula and more general measures $\alpha$.  On the other hand, in these works the nodes used in the quadrature are not restricted to the segment $[-1,1]$. However, the most
important qualitative improvement is that our bound does not tend to $0$ for $c\to 1^+$.

To the best of our knowledge, the only   similar result, i.e.\ the fact that the lower bound does not go to $0$ when the ellipse shrinks
to $[-1,1]$, has been established by Osipenko~\cite{O95} for a very particular weight function, namely the Chebyshev weight function.


Let us describe briefly the content of the paper. In Section~\ref{sec:lower-upper-bnds} we discuss in detail the results of Bakhvalov and Petras concerning
the lower bounds for the integration error for arbitrary quadrature and the upper bounds for the error of the Gauss-Legendre quadrature, and we compare them.
In Section~\ref{sec:new-lower-bound} we develop a new lower bound for the error of an arbitrary quadrature.

\section{Existing error bounds for quadratures of analytic functions}
\label{sec:lower-upper-bnds}

\subsection{Bakhvalov's lower bound for quadratures of analytic functions}
\label{subsec:bach}

The following theorem has been proven in \cite[Thm. 1]{B} (as an
improvement of a previous result from \cite{S}).
\begin{theorem}
\label{thm:bach1}

Assume that $d \alpha=p(x) dx$ and there exists a polynomial $t(x)$ such that $p(x)/t(x) \geq \eta > 0$ for $x \in [-1,1]$.

Let $z_1,z_2,\dots,z_n \in \mathbb{R}$ ($n \leq N$) and let $z_{n+1},\dots,z_N\in \CC$ be points contained
in upper half-plane ($\im z >0$). Let $\mathcal{E}_c$ be an ellipse which encloses all of these points.

For any quadrature formula of the form
\begin{eqnarray}
  Q_N (f)= \sum_{j=1}^n \left(b_{1j} \re f(z_j) + b_{2j} \re f'(z_j) + b_{3j} \im f(z_j) + b_{4j} \im f'(z_j)\right) +  \nonumber \\
      \sum_{j=n+1}^N \left(b_{1j} \re f(z_j) + b_{2j} \re f'(z_j) + b_{3j} \im f(z_j) + b_{4j} \im f'(z_j)\right),
\end{eqnarray}
any $c > 1$ and $M>0$, there exists a function $f_0 \in \A_0(\mathcal{E}_c,M)$ such that
\begin{equation*}
  I(f_0) - Q_N(f_0) \geq \kappa_0 M c^{-2N},
\end{equation*}
where $\kappa_0$ depends on $c$ and the weight function $p(x)$ only.
\end{theorem}

\textbf{Comment:}
\begin{itemize}
\item In terms of the notions introduced earlier, for $N = n$  we have
\begin{equation}
  \rho_n(\A(D),d\alpha) \geq \kappa_0 c^{-2n}.
\end{equation}

\item  In \cite{B} the following formula for
$\kappa_0$ is given (see page 67)
\begin{equation}
\kappa_0=\pi P_0 (1-c^{-1})c^{-2m}(\sinh h)^m,
\label{eq:kappa0-corrected}
\end{equation}
where $h= \ln c$ (hence $\sinh h = (c - c^{-1})/2$) and
constants $P_0\in \mathbb{R}_+$, $m \in \mathbb{N}$   depend on
the weight function only ($P_0$ appears as $Q_0$ in \cite{B}).
 In fact, \cite{B} misprints the formula for $\kappa_0$ as $(1-c^{-1})^{-1}$ instead of
$(1-c^{-1})$.

The constants $m$ and $P_0$ are determined as follows: after the substitution $x=\cos u$ we
have
\begin{equation}
  I(f)=\int_0^\pi f(\cos u) q(u) du, \qquad q(u)=p(\cos u) \sin u.
\end{equation}
Under the assumptions of the theorem the following holds true
\begin{equation}
  q(u)=P(u) l(\cos u),
\end{equation}
where $l$ is a polynomial of degree $m$ and $P(u)=q(u)/(l(\cos u)) \geq P_0 >0$ for $u \in [0,\pi]$ ($P$~appears as $Q$ in \cite{B}).
Therefore, $m$ is the number of zeros in $q(u)$ counted with multiplicities. It is related to the number of zeros in the
weight function $p(x)$: it is the number of zeros  $p(x)$ counted with multiplicities plus two if the  zeros at $0$ and $\pi$
introduced in $q(u)$ by the factor $\sin u$ are not canceled by the singular behavior of $p(x)$, when $x \to \pm 1$. Such cancelations  happen for the
Chebyshev weight (see below).

\begin{itemize}
\item
 For $p(x)\equiv 1$ we have $q(u)=\sin u$. Therefore  $l(z)=1-z^2$, $$\frac{q(u)}{l(\cos u)}=\frac{\sin u}{1-\cos^2 u }= \frac{1}{\sin u} \geq P_0=1$$ for
  $x \in [0,\pi]$. Hence $m=2$.

Easy computations show that for $m=2$ and $P_0=1$ we obtain
\begin{eqnarray}
 \kappa_0 & = & \frac{\pi}{4} (1-c^{-1}) c^{-4} (c-c^{-1})^2
 =   \frac{\pi}{4}(c-1)^3  \frac{(c+1)^2}{c^7}=  \label{eq:ko-bach}\\
& = & \pi (c-1)^3 + O((c-1)^4), \quad  \mbox{ for $c\to 1^+$}. \nonumber
\end{eqnarray}

\item For the Chebyshev weight
$p(x)=1/\sqrt{1-x^2}$
we have
\begin{equation*}
  q(u) = p(\cos u) \sin u = 1.
\end{equation*}
Hence $m=0$ and $P_0=1$, and consequently
\begin{eqnarray}
 \kappa_0=\pi (1-c^{-1}) =\frac{\pi (c-1)}{c}.
\end{eqnarray}
\end{itemize}

We obtain a counter-intuitive statement that when $c-1$ is small (i.e.\  the integrated function is difficult to calculate due to the possible presence of
singularities nearby),
the lower bound for the error is also small. Hence the
quality of the bound is rather poor and can be considerably
improved.

\item $\rho_n(\A(D),d\alpha)$ is  estimated  as follows. For  for any $n$ nodes of   polynomial $f_0 \in \A(\mathcal{E}_c,1)$ of degree $2n+m$ is defined, such that its quadrature error is bounded from below by $\kappa_0 c^{-2n}$. For fixed set of nodes this polynomial is the same for all $c>1$ up to a  multiplicative constant depending on $c$.  Therefore, the functions considered have  no singularities outside the ellipse.

\end{itemize}

\subsection{Petras' lower bounds}
\label{subsec:petras}

 Petras in \cite{P98} considers the quadrature of the same type as in Theorem~\ref{thm:bach1}, ellipses as analyticity regions
and   the Szeg\"o class of weights (measures), which are defined as follows:
$d \alpha(x)=w(x)dx$, where $w$ is a function for which $\int_0^\pi \ln w(\cos x)dx$ exists. It contains the class of weights considered by Bakhvalov.

The reasoning in \cite{P98} goes as follows. Petras proves the following theorem for even more general class of weight measures.
\begin{theorem} \cite[Thm. 2.1]{P98}
\label{thm:PetrasOrtPol}
Assume that the measure $\alpha$ is supported on at least $n+1$ points. Let $D$ be a symmetric domain. Let $p_0,p_1,\ldots,p_n$ be the orthonormal polynomials with respect to the measure~$\alpha$. Then
\begin{equation}
  \rho_n(\A(D),d \alpha) \geq k_n(\A(D),d \alpha):= \left(\sum_{\nu=0}^n (\sup_{z \in D}|p_\nu(z)|)^2\right)^{-1}.
\end{equation}
\end{theorem}

For  weights in the Szeg\"o class and $D=\mathcal{E}_c$ Petras obtains (see Corollary 3.1 in \cite{P98}) the following result
\begin{equation}
  \lim_{n \to \infty} c^{2n} k_n(\A(\mathcal{E}_c), d\alpha)= 2\pi (1-c^{-2}) \cdot \min_{|z|=c} |D(z^{-1})|^2 >0,  \label{eq:Petras-Asymptotic}
\end{equation}
where the so-called Szeg\"o function $D(z)$ is given by
\begin{equation}
  D(z):= \exp \left( \frac{1}{4 \pi} \int_{-\pi}^\pi \frac{1 + ze^{-it}}{1- z e^{-it}} \ln (w(\cos t)|\sin t|) dt \right).
\end{equation}
Observe that apparently from the above formula one obtains $\rho_n \geq O(c-1)/c^{2n}$, the exact form of lower bound depending
on the term involving function $D(z)$ in (\ref{eq:Petras-Asymptotic}) and for $c \geq R > 1$ one obtains
\begin{equation}
\rho_n(\A(\mathcal{E}_c),d\alpha) \geq \kappa_0 c^{-2n},
\end{equation}
generalizing Theorem~\ref{thm:bach1}.

For several particular weights Petras computes an explicit lower bound for $ k_n(\A(\mathcal{E}_c), d\alpha)$,  but it exhibits an incorrect behavior for $c\to 1^+$.

 Below we list only the results for
the Lebesgue measure and the Chebyshev weight.
\begin{itemize}
\item  From Corollary 3.6 in \cite{P98} it follows that for the weight $w(x)\equiv 1$ it holds
\begin{equation}
  \rho_n(\A(\mathcal{E}_c),dx)\geq \pi(1-c^{-2})^2 c^{-2n} \cdot (1 + \eps_n)^{-1},
\end{equation}
where
\begin{equation*}
  0 \leq  \eps_n \leq \frac{c^4 + 4c^2 + 18}{4n c^2 (c^2-1)} + \frac{(n+2)^{3/2}}{c^{n+2}}.
\end{equation*}
It is clear that for a fixed $n$, this bound is $O((c-1)^3)$ for $c \to 1^+$.  To be more precise we have (for fixed $n$)
\begin{equation}
  \rho_n(\A(\mathcal{E}_c),dx) \geq \frac{\pi}{c^{2n}} \left(\frac{32}{23} (c-1)^3 n c^2+O\left((c-1)^4\right)\right).
\end{equation}
\item From Corollary 3.5 in \cite{P98} it follows that for the Chebyshev weight $d\alpha(x)=$ $w(x)dx=$ $dx/\sqrt{1-x^2}$ we have
\begin{equation}
  \rho_n(\A(\mathcal{E}_c),d \alpha) \geq \frac{\pi (1-c^{-2})^3}{2 c^{2n}} \left( 1- \frac{(2n+3)(c^2 -1)+ c^{-2n - 2}}{c^{2n +4}} \right)^{-1} \geq  \frac{\pi (1-c^{-2})^3}{2 c^{2n}}.
\end{equation}
For $c\to 1^+$ we obtain the following estimate  
\begin{equation}
  \rho_n(\A(\mathcal{E}_c),d \alpha) \geq \frac{\pi}{c^{2n}} \left( \frac{3  }{2 n^3+9 n^2+13 n+6}+\frac{6   (c-1) n}{2 n^3+9 n^2+13
   n+6}+O\left((c-1)^2\right) \right).
\end{equation}
In this case for fixed $n$ the lower bound for $c\to 1^+$ does not go to zero, however it goes when $n \to \infty$, which turns out to be unsatisfactory.
\end{itemize}

In fact the bound in Theorem~\ref{thm:PetrasOrtPol} obtained by considering polynomials of degree  $2n$. Hence the functions producing this bound have  no singularities outside the ellipse.

\subsection{Osipenko estimates}
\label{subsec:Osip}
 Osipenko in \cite[Thm. 6]{O95} obtained the following explicit estimate for the Chebyshev weight $d\alpha(x)=dx/\sqrt{1-x^2}$
\begin{equation}
\rho_n(\mathcal{A}(\mathcal{E}_c),d\alpha)=\frac{2\pi}{c^{2n}} +O(c^{-6n}) \label{eq:Osip-estm}
\end{equation}
and  the limit behavior
\begin{equation}
\lim_{c \to 1^+} \rho_n(\mathcal{A}(\mathcal{E}_c),d\alpha)=2\pi.
\end{equation}
 Osipenko uses transformation of an ellipse to an infinite strip, which transforms the problem of integration of bounded analytic functions defined on the ellipse with the Chebyshev weight to the problem of  integration of analytic periodic functions with the Lebesgue measure. He uses  Blaschke products to find  lower estimate for the error, which is natural for this kind of problem. This should be contrasted with the polynomials used to derive lower bounds in \cite{B,P98}.

\subsection{Final comments on Bakhvalov's and Petras' lower bounds}

Both Bakhvalov and Petras  mention that  the Riemann mapping theorem allows to transport the results for an ellipse to other domains. However, no  quantitative statements related to the geometry of the domain $D$ are given.

As it was mentioned in the introduction we have found the behavior of $\kappa_{l}(c)$ for
$c \to 1^+$ obtained by Bakhvalov and by Petras overly pessimistic. In the argument below we will show how bad  this
 bound is qualitatively. Namely, if $\kappa_{g}(c)$ were of the same order as $\kappa_{l}(c)$, i.e.\ $\lim_{c \to 1^+}\kappa_g(c)=0$, the quadrature would be exact even for $n = 1$. This is formalized in the following remark.

\begin{remark}
\label{rem:wrong-lower-bnd}
Let $Q\in \Q(n,\R)$ and a positive bounded function $\kappa:(1,\infty) \times \NN \to \mathbb{R}_+$ be such that
 \begin{equation}
   |I(f) - Q(f)| \leq M \kappa(c,n) c^{-2n}, \quad f \in
   \A_0(\mathcal{E}_c,M). \label{eq:err-estm-hyp}
 \end{equation}
 Assume that  for each $n \in \NN$ holds
\begin{equation}
  \lim_{c \to 1^+}\kappa(c,n)=0.  \label{eq:kappa-lim}
\end{equation}
Then for any $M>0$, $c > 1$, $n \in \NN$ and $f \in
   \A_0(\mathcal{E}_c,M)$ holds
\begin{equation*}
  I(f)=Q(f).
\end{equation*}
\end{remark}

\begin{proof}
 Since $\mathcal{E}_c \subset \mathcal{E}_{c_1}$
for $c<c_1$, we have
\begin{equation}
  \A_0(\mathcal{E}_{c_1},M) \subset \A_0(\mathcal{E}_{c},M), \quad c <
  c_1.  \label{eq:ellipse-incl}
\end{equation}
The above inclusion holds in the following sense: for a function
$f \in \A_0(\mathcal{E}_{c_1},M)$ we consider its restriction to
$\mathcal{E}_c$. It is immediate to see that $f_{|\mathcal{E}_c}
\in \A_0(\mathcal{E}_{c},M)$.

Let us fix $n$ and take a function $f \in
\A_0(\mathcal{E}_{c_1},M)$. By (\ref{eq:err-estm-hyp}) and
(\ref{eq:ellipse-incl})
\begin{equation*}
  |I(f) - Q(f)| \leq M \kappa(c,n) {c}^{-2n}, \quad 1<c\leq c_1.
\end{equation*}
Passing to the limit $c \to 1$ we obtain
\begin{equation*}
  |I(f) - Q(f)| =0.
\end{equation*}
\qed
\end{proof}

\subsection{Upper bounds for Gauss-Legendre quadratures}
\label{secsub:Gauss-quadrature}

We assume that $d\alpha(x)=dx$ and $G_n(f)$ denotes the Gauss-Legendre quadrature with $n$ nodes on $[-1,1]$.

Let us define
\begin{equation}
  r_n(c)=\rho(G_n, \A_0(\mathcal{E}_c,1),dx)=\sup_{f \in \A_0(\mathcal{E}_c,1) }|I(f) - G_n(f)|.
\end{equation}
Obviously
\begin{equation}
  |I(f) - G_n(f)| \leq M r_n(c) , \quad f \in \A_0(\mathcal{E}_c,M).
\end{equation}

Let us list two estimates for the error of Gauss quadrature known in the literature.

Let us start with the estimates for the error of the Gauss
quadrature  due to Rabinowitz \cite[eq.~(18)]{R}, see also
\cite[Thm.~90]{Br} and \cite[Thm.~4.5]{T}
\begin{theorem}
\label{thm:rabinowitz}
\begin{equation}\label{eq:gauss-upper-bnd}
 r_n(c) \leq \min \left(4,\frac{64 }{15(1 - c^{-2})}c^{-2n}\right).
\end{equation}
\end{theorem}
The non-constant part of this estimate has an undesirable property. For $c \to 1$ it explodes,
 which may lead to non-uniform estimates in some contexts.

The bounds which are much more uniform in $c$ for $c \to 1$  are given by Petras in \cite{P95}.
\begin{theorem}  \cite[Thm. 4]{P95}
\label{thm:gauss-petras}
\begin{equation*}
 r_n(c) \leq \frac{4}{c^{2n}}  \left( 1 + \frac{3}{2n c^2}+ \frac{4}{c^{n+1}}\right).
\end{equation*}
\end{theorem}
In fact \cite[Thm. 4]{P95} contains four estimates for $r_n(c)$, such that their mutual ratios are bounded. Here we chose the one, which appears the easiest to handle.

 From Theorem~\ref{thm:gauss-petras} one can easily obtain the following Corollary.
 \begin{cor}
 \begin{eqnarray}
  r_n(c) &\leq&   \frac{26}{c^{2n}}, \label{eq:GL-error-26} \\
  \forall \eps >0  \quad \exists c_0(\eps) \quad \forall c \geq c_0(\eps) \quad  r_n(c) &\leq&   \frac{4 + \eps}{c^{2n}}. \label{eq:GL-error-4-eps}
\end{eqnarray}
\end{cor}

\begin{remark}\label{rmk:chebyshev-2n}
In \cite{P95} (in part (b) of a remark just below Theorem 4 there) Petras mentions  that taking $f$ to be a suitably scaled ($2n$)-th Chebyshev polynomial of the first kind $T_{2n}$, i.e.\ $f=\frac{2c^{2n}}{c^{4n}+1} T_{2n} \in \A_0(\mathcal{E}_c,1)$ one obtains

\begin{equation}
    |I(f) - G_n(f)|  \geq \frac{\pi (1 - (4n)^{-1})}{c^{2n}(1 + c^{-4n})}.
\end{equation}
Hence, the bounds given in Theorem~\ref{thm:gauss-petras} are optimal, up to a constant independent of $c$ and $n$.
\end{remark}

Observe that from (\ref{eq:GL-error-26}) it follows that if $M/\eps > 26$, then in order to have the error less than $\eps$ for
functions from $\A_0(\mathcal{E}_c,M)$ it is enough to use $N_g$ nodes, where
\begin{equation}
  N_g \geq \frac{\ln \frac{M}{\eps}}{\ln c}.  \label{eq:N-GL-nodes}
\end{equation}

\subsection{Comparison of lower and upper bounds}
\label{subsec:lo-upp-bnds}

We are now ready to compare in detail the lower bounds of Bakhalov and Petras with the bounds for the Gauss-Legendre quadrature for the ellipses with the Lebesgue measure as the weight function.

Let $c>1$ and let $\kappa_l(c)$ and $\kappa_g(c)$ be positive numbers  such that
\begin{itemize}
\item for any $Q_n\in \Oq(n,\R_2)$ there is an  $f_0 \in  A_0(\mathcal{E}_c,M)$ such that
\begin{equation}
|I(f_0) - Q_n(f_0)| \geq M \kappa_{l}(c) c^{-2n},  \label{eq:lo-bnd-err}
\end{equation}

\item for the Gauss-Legendre quadrature $G_n$, for any  $f  \in  A_0(\mathcal{E}_c,M) $ we have
\begin{equation}
 |I(f) - G_n(f)| \leq M \kappa_g(c) c^{-2n}, \label{eq:up-bnd-err}
\end{equation}
\end{itemize}
where $\kappa_l$ is Bakhvalov's or Petras' lower bound discussed in Sections \ref{subsec:bach} and \ref{subsec:petras} and $$\kappa_g(c)=\sup_{n \geq 1}c^{2n}r_n(c)$$ obtained
from Theorem~\ref{thm:rabinowitz} or Theorem~\ref{thm:gauss-petras}.

From Theorem~\ref{thm:bach1} (with
$\kappa_l=\kappa_0$ given by (\ref{eq:ko-bach}))  for $c$ close to $1$ we get
\begin{eqnarray*}
\kappa_{l}(c) & = & (c-1)^3\pi +O((c-1)^4),   \\
\kappa_{g}(c) & = & 26.
\end{eqnarray*}
For large values of  $c$ (which means that we are considering very regular functions) we have from~(\ref{eq:GL-error-4-eps})
\begin{eqnarray*}
\kappa_{l}(c) & = & \frac{\pi}{4}c^{-2} + O(c^{-4}),  \\
\kappa_g(c) &=& 4.1 + O(c^{-2}).
\end{eqnarray*}

Note that in both cases the quotient
$\kappa_{g}/\kappa_{l}\to \infty$.

Both bounds~(\ref{eq:lo-bnd-err}) and~(\ref{eq:up-bnd-err}) are $O(c^{-2n})$ as the
function of $n$ and they give the following estimates for
$n$, needed to obtain the integral with error less
than $\eps$.

For the Gauss-Legendre quadrature it is enough to take $n \geq
N_g$, where
\begin{equation*}
  N_g=N_g\left(\frac{M}{\eps},c\right)=\max\left(1,\frac{1}{2 \ln
  c}\ln\left(\frac{M}{\eps}\kappa_g(c)\right)\right),
\end{equation*}
while (\ref{eq:lo-bnd-err}) implies that whatever the quadrature
is we cannot take $n$ smaller than
\begin{equation*}
  N_l=N_l\left(\frac{M}{\eps},c\right)=\max\left(1,\frac{1}{2 \ln
  c}\ln\left(\frac{M}{\eps}\kappa_l(c)\right)\right).
\end{equation*}

For $\eps \to 0$ we have
\begin{equation*}
\frac{N_l}{N_g}\approx
\frac{\ln\left(\frac{M}{\eps}\kappa_l(c)\right)}{\ln\left(\frac{M}{\eps}\kappa_g(c)\right)} \to 1.
\end{equation*}
Apparently both numbers $N_l$ and $N_g$ are of similar magnitude up to
a  factor depending on $c$ but not on $n$.

However, if we fix $\eps$ and let $c \to 1$, we have $\kappa_l(c) \to 0$, hence $N_l \to 1$ the lower bound $N_l$ loses its predictive power.

We are not concerned with the behavior of $\kappa_l$ and $\kappa_g$ for $c \to \infty$, because it does not necessarily make sense to increase $c$ while
keeping $M$  constant; the functions in
$\A_0(\mathcal{E}_c,M)$ become very flat for  large  $c$ and in this limit we obtain $N_l=N_g=1$.

Summing up, the bounds (\ref{eq:lo-bnd-err}) and (\ref{eq:up-bnd-err})  might give completely different estimates $N_l$ and $N_g$ of information needed to bring the error below $\eps$. For `difficult' functions ($c$ close to 1) we obtain the obvious bound $n\ge N_l = 1$ for a significant range of the ratio $M/\eps$.

  It appears to us that it makes sense to require the following condition to maintain the optimality of Gauss-Legendre quadratures on ellipses:
   there exists $\eta_0$ such that for all $M/\eps \in \mathbb{R}_+$ and $c>1$
\begin{equation}
0 < \eta_0 \le \frac{N_l\left(\frac{M}{\eps},c\right)}{N_g\left(\frac{M}{\eps},c\right)}.  \label{eq:optimality-crit}
\end{equation}
Observe that, when compared to Definition~\ref{def:petras-loss}, we  now  want
the ratio to be bounded also when we change the ellipse.

\section{New lower bounds }\label{sec:new-lower-bound}
 In this section we study the problem of estimating from below the quadrature error
  in a class of  analytic functions with possible singularities outside
  a nice domain. In the special
   case of ellipses the formulas are given so that they can be directly compared with the
   known ones. Since the methods may probably be applicable in a more general class
   of domains (not necessarily simply connected) we introduce distances (metrics) that could
   be tools for studying them in several complex variables. But we restrict
    our consideration to the case of simply connected domains in the complex plane
    where the considered (hyperbolic) metric and distance may be described in many
    equivalent ways. The question which description could (and should) be applied in
    the case of domains being not simply connected remains open.


\subsection{Definitions and description of the problem}

By $\lambda_1(A)$ we denote the Lebesgue measure of the set $A \subset \mathbb{R}$.

We recall that \textit{the Poincar\'e distance $p$ } on the open unit disk $\DD:=\{z\in\mathbb C:|z|<1\}$ is given by the formula
\begin{equation}
p(z,w):=\frac{1}{2}\ln\frac{1+m(w,z)}{1-m(w,z)}=:\arctanh(m(w,z)),\; w,z\in\DD,  \label{eq:def-pzw}
\end{equation}
where $m(w,z)=\left|\frac{w-z}{1-\bar wz}\right|$.

The Poincar\'e distance induces the pseudodistance $c_D$ on any domain (i.e.\ connected and open set)
$D\subset\CC$ by the following formula
\begin{equation}
c_D(w,z):=\sup\{p(F(w),F(z)):F\in\OO(D,\DD)\}, w,z\in D,
\end{equation}
where $\OO(D,\DD)$ denotes the set of holomorphic (analytic) functions $D$ to $\DD$.
We also put
\begin{equation}
c_D^*(w,z):=\tanh c_D(w,z).
\end{equation}

We remind the following property of $c_D$ (called \textit{the holomorphic contractibility of $c$}):
$c_G(F(w),F(z))\leq c_D(w,z)$ for any $F\in\OO(D,G)$, $w,z\in D$.
In the case of simply connected domains the function $c_D$ coincides with the distance induced by the metric $\gamma_D$ (often called \textit{hyperbolic metric} for planar domains) defined by the formula
\begin{equation}
\gamma_D(z;X):=\sup\{|F^{\prime}(z)X|/(1-|F(z)|^2):F\in\OO(D,\DD)\},\; z\in D,\; X\in\CC.  \label{eq:gammaD}
\end{equation}
It is well-known that $\gamma_{\mathbb D}(z;X)=|X|/(1-|z|^2)$, $z\in\mathbb D$, $X\in\mathbb C$ (we call the function $\gamma_{\mathbb D}$ \textit{the Poincar\'e metric}).

Similarly as before we get a version of holomorphic contractibility of $\gamma$, namely the inequality
\begin{equation}
\gamma_G(F(w);F^{\prime}(w)X)\leq\gamma_D(w;X),\; w\in D;X\in\mathbb{C},  \label{eq:gamma-ineq}
\end{equation}
for any $F\in\OO(D,G)$. For domains $D\subset G$ in $\mathbb C$ we may use the holomorphic contractibility for the inclusion function $\iota:D\mapsto G$ where $D\subset G\subset\mathbb C$ which gives, among others, the inequality $\gamma_D(z;1)\geq\gamma_G(z;1)$, $z\in D$.

Note that although we defined
the functions $c_D$ and $\gamma_D$ in a very general situation we shall consider them in the very special case of $D$ being a simply connected domain.

The geometry induced by the Poincar\'e distance is an example of a
non-Euclidean geometry. Recall that the lines (geodesics) in this
geometry are diameters and the arcs of circles lying in $\DD$ and
being orthogonal to the unit circle $\partial\DD$. In particular,
for three consecutive points $x,y,z$ on such geodesics one has the
equality $p(x,z)=p(x,y)+p(y,z)$. Note also that the biholomorphic
mappings transform geodesics into geodesics, and the geodesics in the domain $D$ satisfy the equality $c_D(x,z)=c_D(x,y)+c_D(y,z)$ for three consecutive points lying in the geodesic.
The distance of two points $w,z$ from the simply connected domain $D$ lying in a geodesic may be given with the help of the function $\gamma_D$ as follows. If $\alpha:[0,1]\to D$
is a parametrization of the part of the geodesic joining $w$ and $z$ lying between $w$ and $z$; then
\begin{equation*}
c_D(w,z)=\int\sb{0}\sp{1}|\alpha^{\prime}(t)|\gamma_D(\alpha(t);1)dt.
\end{equation*}
We should also keep in mind that the Poincar\'e distance on $\mathbb D$ (as well as the Poincar\'e metric) are invariant under holomorphic automorphisms of the unit disk ($\operatorname{Aut}(\mathbb D)$). Recall that
\begin{equation}
\operatorname{Aut}(\mathbb D)=\left\{e^{i\theta}m_{\eta}:\theta\in\mathbb R,\eta\in\mathbb D\right\},
\end{equation}
where $m_{\eta}(z):=(\eta-z)/(1-\bar{\eta} z)$, $z\in\mathbb D$.

A special role in our considerations will be played by \textit{the finite Blaschke products}.
Some of basic properties of the finite Blaschke products are that they extend holomorphically to a neighborhood of $\overline{\mathbb D}$ (they are rational with poles lying outside of the closed unit disk. The finite Blaschke product $B$ is a proper holomorphic mapping of $\mathbb D$ onto $\mathbb D$. Moreover, $|B(z)|=1$, $|z|=1$.

We refer the reader to any of the textbooks
\cite{Rud 1966}, \cite{Con 1978},
\cite{Con 1995} and
\cite{Jar-Pfl 1993}. In the last reference the theory of
holomorphically invariant metrics and distances in several complex
variables is presented.

In higher-dimensional case the metric $\gamma_D$ depends
on points $z\in D$ and the vectors $X$ from the tangent space to
$D$; that is the reason why the value of the differential at
vector $X\in\mathbb{C}$ (generally $\mathbb{C}^n$) is studied.
However, the facts that we use are standard in the theory of one
complex variable and may be found in many textbooks on the theory
of complex variable. As to the theory of (bounded) holomorphic
functions, except for the above mentioned textbooks, we refer the reader to \cite{Dur 1970}, \cite{Gar 1981} (where one may also
see how the Blaschke products appear naturally when considering some extremal problems in the theory of analytic functions). Out of many
possible references for the properties of the Carath\'eodory
distance (induced by the hyperbolic metric) we recommend the paper
\cite{Ban-Car 2008}  and the references therein concerning
estimates for the hyperbolic metric in the ellipses.
Note that the hyperbolic density $\sigma_D$  considered in \cite{Ban-Car 2008}
is related to $\gamma_D$ by the relation
$\gamma_D(z;X)=|X|\sigma_D(z)$. The paper \cite{Ban-Car 2008}
could also possibly be applied to sharpen some of the results
presented in the paper in the case of ellipses.

In this section, unless otherwise stated, the domain  $D\subset\mathbb C$ contains $[-1,1]$ is simply connected, $D\neq\CC$ and $D$ is
symmetric with respect to the $x$-axis, i.e.\ $z\in D$ iff $\bar z\in
D$. Let $\alpha$ be a finite, positive, Borel measure on $[-1,1]$ absolutely continuous with respect to the Lebesgue measure.

Let $f_D:D\to\DD$ be a conformal mapping  (i.e.\ biholomorphic) such that $f_D(0)=0$,
$f_D([0,1])\subset[0,1)$ (the latter is possible because of the
symmetry of $D$). Note also that the function $f_D$ is defined is actually unique
(it follows from the uniqueness part of the Riemann mapping
theorem). The set $\mathbb R\cap D$ is a geodesic. We shall often make use of the identity
\begin{equation*}
 c_D^*(w,z)=m(f_D(w),f_D(z)),\;w,z\in D.  \label{eq:cd-wzor}
\end{equation*}

Given an integer $k$ let $r(k)$ be the least even integer bigger than or equal to $k$.
Certainly, $r(k)$ is either $k$ or $k+1$.

For the sequence of $n$  distinct points $X:=(x_1,\ldots,x_n)$
where $-1\leq x_1<\ldots<x_n\leq 1$, the sequence of $n$ positive integers
$\mathcal K=(k_1,\ldots,k_n)$ we define
\begin{equation*}
\mathcal F(D;X;\mathcal K):=\{f\in\OO(D,\DD):f^{(l)}(x_j)=0:l=0,\ldots,k_j-1;\;j=1,\ldots,n\},
\end{equation*}
\begin{equation*}
\mathcal F_r(D;X;\mathcal K):=\{f\in\mathcal F(D;X;\mathcal K):f(D\cap\mathbb R)\subset\mathbb R\},
\end{equation*}
\begin{equation*}
\mathcal F_+(D;X;\mathcal K):=\{f\in\mathcal F_r(D;X;\mathcal K):f\geq 0 \text{ on } D\cap\mathbb R\}
\end{equation*}
and
\begin{equation*}
J_a(D;X;\mathcal K):=\sup\left\{\left|\int\sb{-1}\sp{1}g(x) d\alpha(x) \right|:g\in \mathcal F_a(D;X;\mathcal K)\right\},
\end{equation*}
where $a$ is $+$, $r$ or empty sign.

We are now in a position to prove the following lemma.
\begin{lemma}\label{lem:unique-extremal}
Let $D$, $f_D$, $\alpha$, $X$ and $\mathcal K$  be defined as above. Then there is exactly
one $f\in \mathcal F_+(D;X;\mathcal K)$ such that
\begin{equation*}
\int\sb{-1}\sp{1}f(x)d\alpha(x)=J_+(D;X;\mathcal K).
\end{equation*}
Moreover, $f$ is given by the formula
\begin{equation}
f(z)=\prod\sb{j=1}\sp{n}\left(\frac{f_D(z)-f_D(x_j)}{1-f_D(x_j)f_D(z)}\right)^{r(k_j)}, \; z\in D
  \label{eq:func-extremalna}
\end{equation}
and
\begin{equation*}
J_+(D;X;\mathcal K) = \int\sb{-1}\sp{1}\left(\prod\sb{j=1}\sp{n}(c_D^*(x,x_j))^{r(k_j)}\right)d\alpha= \int\sb{-1}\sp{1}\left(\prod_{j=1}^n
m(f_D(x),f_D(x_j))^{r(k_j)}\right)d\alpha.
\end{equation*}
\end{lemma}

\begin{proof}
Let $g\in \mathcal F_+(D;X;\mathcal K)$.
 The non-negativity of $g$ together
with the vanishing of derivatives at $x_j$ implies that the multiplicity
of $g$ at $x_j$ is at least $r(k_j)$.  Let $f$ be the function given by the formula (\ref{eq:func-extremalna}). Then the function $h:=\frac{g}{f}$
extends to a well-defined holomorphic function on $D$.
Moreover, the function $f$ is the composition of the finite Blaschke product with the conformal function $f_D$ so  $\lim\sb{z\to\partial D}|f(z)|=1$ and thus $\limsup\sb{z\to\partial D}|h(z)|\leq 1$.
This  together with the maximum principle for holomorphic functions implies that $|h(z)|\leq 1$, $z\in D$.
Additionally, the maximum principle gives that the equality at one point $z\in D$ holds iff $h$ is constant. And the
non-negativity of $f$ and $g$ on $[-1,1]$ implies that this constant is one.
Consequently, either $g(z)=f(z)$, $z\in D$ or $|g(z)|<|f(z)|$,
$z\in D\setminus \{x_1,\ldots,x_n\}$, which completes the proof.
\qed
\end{proof}

\begin{remark} It is obvious that
\begin{equation}
J_+(D;X;\mathcal K)\leq J_r(D;X;\mathcal K)\leq J(D;X;\mathcal K).
\end{equation}
Moreover, the second inequality above is actually the equality. To see this take any $g\in\mathcal F(D;X;\mathcal K)$. Let $|\omega|=1$ be such that $\omega\int\sb{-1}\sp{1}g(x)dx=\left|\int\sb{-1}\sp{1}g(x)dx\right|$. Define
$h(\lambda):=(\omega g(\lambda)+\overline{\omega g(\overline{\lambda})})/2$, $\lambda\in D$.
Then $h\in\mathcal F_r(D;X;\mathcal K)$ and $h(x)=\re (\omega g(x))$, $x\in[-1,1]$. Consequently,
\begin{equation}
\left|\int\sb{-1}\sp{1}g(x)dx\right|=\re\left(\omega\int\sb{-1}\sp{1}g(x)dx\right)=\int\sb{-1}\sp{1}h(x)dx,
\end{equation}
which implies the inequality $ J(D;X;\mathcal K)\leq J_r(D;X;\mathcal K)$.

On the other hand $J_+(D;X;\mathcal K)$ is, in general, less than $J_r(D;X;\mathcal
K)$.  It can already be seen when considering $n=1$, $k_1=1$, $d\alpha(x)=dx$ and
$x_1$ close to $-1$. In fact, first note that for $x_1=-1$ we get the inequalities
\begin{equation}
1>\frac{f_D(x)-f_D(x_1)}{1-f_D(x_1)f_D(x)}>\left(\frac{f_D(x)-f_D(x_1)}{1-f_D(x_1)f_D(x)}\right)^2>0,\;x\in(-1,1]
\end{equation}
so the inequality
\begin{equation}
\int\sb{-1}\sp{1}\left(\frac{f_D(x)-f_D(x_1)}{1-f_D(x_1)f_D(x)}\right)dx >
\int\sb{-1}\sp{1}\left(\frac{f_D(x)-f_D(x_1)}{1-f_D(x_1)f_D(x)}\right)^2dx
\end{equation}
holds for $x_1\geq -1$ sufficiently close to $-1$.
\end{remark}

\begin{remark}
 Recall that the finite Blaschke products are extremal in many problems which involve bounded holomorphic functions on the unit disk. In the context
 of the optimal quadrature formula the Blaschke products have been used by Osipenko \cite{O95} and Bojanov \cite{Bo73,Bo74} for the analytic functions on the unit circle.  Therefore, it is very natural that the function for which the supremum in Lemma~\ref{lem:unique-extremal} is attained is, up to a conformal mapping $f_D$, a finite Blaschke product.
\end{remark}

In the next subsection
we shall estimate from below the number
\begin{eqnarray*}
J_+(D;N) & := & \inf\{J_+(D;(x_1,\ldots,x_n);(k_1,\ldots,k_n)):\\
  & & \mbox{}\hspace*{3cm} n\in\mathbb N, -1\leq x_1<\ldots<x_n\leq 1,k_1+\ldots+k_n=N\}.
\end{eqnarray*}



\subsection{Lower estimate}

First we recall the classical Koebe one-quarter theorem.
\begin{theorem}
\label{thm:Koebe} \rm{ (see e.g. \cite{Con 1995}, Thm 14. 7. 8)}
The image of an injective holomorphic function $f: \DD \to  \mathbb{C}$  contains the disk centered at $f(0)$ with radius $|f'(0)|/4$.
\end{theorem}
Before we proceed further with estimates for nice domains we present a result on a more general class of domains. First we remind that for any domain $D\subset\mathbb C$, $D\neq \mathbb C$ we define
$\delta_D(x):=\inf\{|x-z|:z\in\CC\setminus D\}$, $x\in D$.

\begin{lemma}\label{lemma:distance} Let $D$ be a simply connected domain in $\mathbb C$, $D\neq\mathbb C$ (we do not assume the symmetry of $D$!). Let $z_0\in D$. Then
$\gamma_D(z_0;1)\geq\frac{L}{\delta_D(z_0)}$ where $L=1/4$. If $D$ is additionally convex then we may take in the inequality $L=1/2$.
\end{lemma}
\begin{proof} Let $g:\mathbb D\to D$ be the conformal mapping such that $g(0)=z_0$. Applying Theorem~\ref{thm:Koebe} to
$g$ we get that $\delta_D(z_0)\geq|g^{\prime}(0)|/4$. But then $\gamma_D(z_0;1)\geq \left|\left(g^{-1}\right)^{\prime}(z_0)\right|=1/|g^{\prime}(0)|$ which finishes the proof in the general case.

Assume now additionally that $D$ is convex.
 Then after translating and rotating the set $D$, we can assume that $D\subset H:=\{\operatorname{Re}z>0\}$ and $z_0=\delta_D(z_0)$.
Define the biholomorphism $F:H \to \mathbb{D}$, $F(z) = (z-1)/(1+z)$. From (\ref{eq:gamma-ineq}) and (\ref{eq:gammaD}) if follows that
\begin{eqnarray*}
  \gamma_D(z_0;1)\geq\gamma_H(z_0;1)= \frac{|F'(z_0)|}{1 - |F(z_0)|^2}.
\end{eqnarray*}
Taking into account that $z_0 =\delta_D(z_0)>0$ we obtain the following estimate
\begin{equation*}
   \gamma_D(z_0;1) \geq \frac{1}{2z_0}=\frac{1}{2\delta_D(z_0)}.
\end{equation*}
\qed
\end{proof}

Recall now that we assume that $D$ is a simply connected domain, symmetric with respect to the real axis and such that $[-1,1]\subset D\subset\mathbb C$, $D\neq\mathbb C$. We remind that in such a case we define
\begin{equation*}
\delta_D:=\sup\{\delta_D(x):x\in[-1,1]\}.
\end{equation*}
Observe that $\delta_D$ is the radius of the largest disk with the center in $[-1,1]$, which is contained in $D$.

\begin{lemma} For all $w,z\in[-1,1]$ the following inequality holds $c_D(w,z)\geq\frac{L}{\delta_D}|w-z|$, where
$L=1/4$. Moreover, in the case $D$ is additionally convex we may
take $L=1/2$. Consequently,
\begin{equation}\label{eq:koebe}
m(f_D(w),f_D(z))=c_D^*(w,z)=\tanh  c_D(w,z)\geq \frac{\exp\left(\frac{2L|w-z|}{\delta_D}\right)-1}{\exp\left(\frac{2L|w-z|}{\delta_D}\right)+1},\;w,z\in [-1,1].
\end{equation}
\end{lemma}
\begin{proof}


Due to the simple fact that $\mathbb R\cap D$ is a geodesic and applying Lemma~\ref{lemma:distance} we get
\begin{equation*}
c_D(w,z)=\int\sb{0}\sp{1}|w-z|\gamma_D(tw+(1-t)z;1)dt\geq \frac{L|w-z|}{\delta_D}.
\end{equation*}


As to the last inequality in (\ref{eq:koebe}), recall that $\tanh$ is an increasing function so we
obtain
\begin{equation*}
c_D^*(w,z)= \tanh c_D(w,z)\geq \tanh
\frac{L}{\delta_D}|w-z| =
\frac{\exp\left(\frac{2L|w-z|}{\delta_D}\right)-1}{\exp\left(\frac{2L|w-z|}{\delta_D}\right)+1},\;w,z\in
[-1,1].
\end{equation*}
\qed
\end{proof}

Let us prove the  general estimate for $J_+$.
\begin{theorem}\label{thm:estimates}
Given a positive number $N\in\mathbb N$ the following inequality holds
\begin{equation}
   J_+(D;N)\geq \sup_{\eps >0}\left\{\left(\frac{\exp\left(\frac{2L\eps}{\delta_D}\right)-1}{\exp\left(\frac{2L\eps}{\delta_D}\right)+1}\right)^{2N}\left(\alpha([-1,1])-\omega(2N\eps,\alpha)\right)\right\}.
\end{equation}

where $\omega(\delta,\alpha):=\sup\left\{\alpha(A):A\subset [-1,1]\text{ is a Borel subset, } \lambda_1(A)\leq \delta\right\}$.

Moreover,
\begin{equation}
\lim\sb{\delta_D\to 0}J_+(D;N)=\alpha([-1,1]).
\end{equation}
\end{theorem}

\begin{proof} Fix
$\eps>0$. For any compact set $K$ denote  $K^{\eps}:=\{z\in\CC:|z-x|<\eps
\text{ for some $x\in K$}\}$. Denote also  $r:=r(k_1)+\ldots+r(k_n)$.
By decreasing the set of integration, applying Lemma~\ref{lem:unique-extremal} and the estimate (\ref{eq:koebe}), keeping in mind that the integrands take the values in the interval $[0,1)$ we get the following inequality
\begin{equation*}
  J_+(D;(x_1,\ldots,x_n),(k_1,\ldots,k_n))\geq\int\sb{[-1,1]\setminus \{x_1,\ldots,x_n\}^{\eps}}\left(\frac{\exp\left(\frac{2L\eps}{\delta_D}\right)-1}{\exp\left(\frac{2L\eps}{\delta_D}\right)+1}\right)^{r}d\alpha.
\end{equation*}
Since $n\leq N$, we get that $r\leq 2N$ so
\begin{equation*}
  J_+(D;(x_1,\ldots,x_n),(k_1,\ldots,k_n))\geq\left(\frac{\exp\left(\frac{2L\eps}{\delta_D}\right)-1}{\exp\left(\frac{2L\eps}{\delta_D}\right)+1}\right)^{2N}\int\sb{[-1,1]\setminus \{x_1,\ldots,x_n\}^{\eps}}d \alpha.
\end{equation*}
Since $\lambda_1(\{x_1,\ldots,x_n\}^{\eps})\leq 2n\eps\leq 2N\eps$ we conclude the proof of the proposition.
\qed
\end{proof}

Note that Theorem~\ref{thm:estimates} gives essential improvement of the
estimates in \cite{B}, \cite{P98}, $J(D;N)$ is estimated from below by a function
tending to $0$ as $\delta_D\to 0$. Moreover, the estimate in
\cite{B}, \cite{P98} are studied in detail for ellipses only.



\begin{theorem}\label{thm:lower-estimate}
Let $D\subset\mathbb C$ be a domain as above (i. e. simply connected, symmetric with respect to the $x$-axis, $[-1,1]\subset D$, $D\neq\mathbb C$) and let $\alpha=\lambda_1$. Then for any positive integer $N$ we get the following estimate (recall that in general case $L=1/4$ and in the case of $D$ convex $L=1/2$)
\begin{equation*}
J_+(D;N)\geq 2 L^{2N}\frac{\delta_D^{(2N\delta_D)/L}}{(\delta_D+L)^{(2N/L)(\delta_D+L)}}.
\end{equation*}
In the case $D$ is convex the above inequality gives
\begin{equation*}
J_+(D;N)\geq
2\left((1+1/(2\delta_D))^{2\delta_D}(2\delta_D+1)\right)^{-2N}\geq 2 \exp(-2N)(2\delta_D+1)^{-2N}.
\end{equation*}
\end{theorem}
\begin{proof}

Since for $ t\geq 0$
$$\frac{\exp(t)-1}{\exp(t)+1}\geq \frac{t}{2+t}$$ by (\ref{eq:koebe}) and Lemma~\ref{lem:unique-extremal} we get
\begin{multline*}
J_+(D;N)  \geq \\
\inf\left\{\int\sb{-1}\sp{1}\prod\sb{j=1}\sp{n}\left(\frac{L|x-x_j|}{\delta_D+L|x-x_j|}\right)^{r(k_j)}dx:
 n\in\mathbb N,\; -1\leq x_1<\ldots<x_n\leq 1,\;k_1+\ldots+k_n=N\right\}.
\end{multline*}
The Jensen inequality now implies that $J_+(D;N)$ is not less than the infimum of
\begin{equation}\label{eq:jensen2}
2\exp\left(1/2\sum\sb{j=1}\sp{n}r(k_j) \int\sb{-1}\sp{1}\left(\ln(L|x-x_j|)-\ln(\delta_D+L|x-x_j|)\right)dx\right).
\end{equation}
 taken over all sequences $-1\leq x_1<\ldots<x_n\leq 1$, $k_1+\ldots+k_n=N$.




The integral in (\ref{eq:jensen2}) equals
\begin{multline*}
I_j =
 2\ln L+(1-x_j)\ln (1-x_j)+(1+x_j)\ln (1+x_j) + \\
-\left(1/L\right)(L(1-x_j)+\delta_D)\ln (\delta_D+L(1-x_j)) +\\
-\left(1/L\right)(L(1+x_j)+\delta_D)\ln (\delta_D+L(1+x_j))+(2/L)\delta_D\ln \delta_D.
\end{multline*}

We now rewrite it in the form
\begin{equation*}
  I_j=g(x_j) + g(-x_j) + 2\ln L + (2/L)\delta_D\ln \delta_D,
\end{equation*}
where
\begin{equation*}
  g(t)=(1+t) \ln(1+t) -\frac{1}{L}(L(1+t)+ \delta_D) \ln (L(1+t)+ \delta_D), \; t\in[-1,1].
\end{equation*}
By setting $h(t):=g(t)+g(-t)$, $t\in[-1,1]$
we get
\begin{eqnarray*}
   h'(t) & = & g'(t) - g'(-t)\\
  & = & \ln \frac{1+t}{1-t} - \ln\frac{L(1+t)+ \delta_D}{L(1-t)+ \delta_D}\\
  & = & \ln \frac{(1+t)(L(1-t)+ \delta_D)}{(1-t)(L(1+t)+ \delta_D)}.
\end{eqnarray*}

It is clear that $h$ is even and $h'(0)=0$. Moreover, $h'(t)>0$ for $t\in(0,1)$.

Indeed $h'(t)>0$ iff $(1+t)(L(1-t)+ \delta_D) > (1-t)(L(1+t)+ \delta_D)$. This condition is equivalent to
\begin{displaymath}
  L(1-t^2) + (1+t) \delta_D >  L(1-t^2) + (1-t) \delta_D,
\end{displaymath}
which is satisfied for $t>0$.

The above calculations show that the function defined by the formula
\begin{multline*}
(1+t)\ln(1+t)+(1-t)\ln(1-t)-\left(1/L\right)(L(1+t)+\delta_D)\ln(\delta_D+L(1+t))+\\
- \left(1/L\right)(L(1-t)+\delta_D)\ln(\delta_D+L(1-t))
\end{multline*}
attains its minimum on the interval $[-1,1]$ at $t=0$. Since  $r=\sum r(k_j) \leq 2N$ we get
\begin{equation*}
\ln\left(J_+(D;N)/2\right) \geq 2N\left(\ln L-\left(1/L\right)(L+\delta_D)\ln (L+\delta_D)+\left(1/L\right)\delta_D\ln\delta_D\right)
\end{equation*}
and consequently
\begin{equation*}
 J_+(D;N)\geq 2 L^{2N}\frac{\delta_D^{(2N\delta_D)/L}}{(\delta_D+L)^{(2N/L)(\delta_D+L)}}.
\end{equation*}
Note that the last expression  tends to $2$ as $\delta_D\to 0$ (compare Theorem~\ref{thm:estimates}).

On the other hand, in the case when $D$ is convex, we have
\begin{eqnarray*}
J_+(D;N) & \geq & 2 L^{2N}\frac{\delta_D^{(2N\delta_D)/L}}{(\delta_D+L)^{(2N/L)(\delta_D+L)}}\\
  & = & 2\left((1+1/(2\delta_D))^{2\delta_D}(2\delta_D+1)\right)^{-2N}\\
  & > & 2 \exp(-2N)(2\delta_D+1)^{-2N}\label{eqn:low-exp},
\end{eqnarray*}
in view of the inequality  $\left(1+1/x\right)^x < e$ for $x >0$.
\qed
\end{proof}

\textbf{Proof of Theorem~\ref{thm:main}:} Without loss of generality we may assume that $M=1$.  Fix also the nodes $x_j$ and integers $k_j$, $j=1,\ldots,n$. Let $f$ be the unique function for which the supremum in the definition of  $J_+(D;X;\mathcal K)$ is attained (compare Lemma~\ref{lem:unique-extremal}). Since the function $f$ belongs to the class $\mathcal F(D;X;\mathcal K)$, we get $f^{(l)}(z_j)=0$ for $l=0, \ldots, k_j-1;\;j=1,\ldots,n$ and consequently it gives the quadrature $Q$  the same information as does the function $g\equiv 0$. Therefore,
\begin{equation}
  Q(f)=Q(g).
\end{equation}
From Theorem~\ref{thm:lower-estimate} it follows that
\begin{equation}
I(f) \geq 2 \gamma,
\end{equation}
where $\gamma$ is defined by (\ref{eq:gamma})

Since $Q(f)=0$, we immediately get
\begin{equation}
  |I(f) - Q(f)| \geq \gamma.
\end{equation}


\qed

\subsection{The case of ellipses}
In the case when $D$ is an ellipse
\begin{equation*}
\mathcal{E}_c:=\{(x,y)\in\mathbb R^2:x^2/a^2+y^2/b^2<1\},
 \end{equation*}
 $a^2-b^2=1$, $c:=a+b$, $a,b>0$, simple computations lead to the relations $a=(c^2+1)/(2c)$, $b=(c^2-1)/(2c)$ and the formula\\
$$
\delta_{E_C}(x) = \left\{
\begin{array}{ll}
  \sqrt{a^2-1}\sqrt{1-x^2},& x\in [-1/a,1/a],\\
  \min\{|x\pm a|\}, & x\in[-1,1]\setminus(-1/a,1/a).
\end{array}
\right.
$$

Consequently, $\delta_{\mathcal{E}_c}=\sqrt{a^2-1}=(c^2-1)/(2c)$.

Therefore, as an immediate consequence of Theorem~\ref{thm:lower-estimate}, we get the following lower bound in the case of the ellipse and $\alpha$ being the Lebesgue measure.

\begin{cor}
\label{cor:ellipse}
 Let $c>1$. Then
\begin{equation}
J_+(\mathcal{E}_c;N)\geq 2\left(\left(\frac{c^2-1+c}{c^2-1}\right)^{(c^2-1)/c}\left(\frac{c^2-1+c}{c}\right)\right)^{-2N}. \label{eq:low-bound-ell}
\end{equation}
\end{cor}

%

\begin{theorem}
\label{thm:ell-num-nodes-nbd}
Let $Q\in \Oq(n, \R)$, such that $|\R|= N$
be a quadrature on $\mathcal{E}_c$. Then for $c$ close to $1$, in order to have the error of $Q$ smaller than $\eps$, $N$ has to be greater than
\begin{equation}
 N_l\left(\frac{M}{\eps},c\right)=  \frac{-\ln \frac{M}{\eps}}{4 (c-1) \ln(c-1)} \left(1 + O\left(\left|\frac{1}{\ln(c-1)}\right|\right)\right).
\end{equation}
\end{theorem}
\proof
Let us remind the reader that for all functions appearing in the definition of $J_+(\mathcal{E}_c;N)$ we have had a bound $|f(z)| \leq 1$ for $z \in \mathcal{E}_c$.

Therefore from (\ref{eq:low-bound-ell}) (see also the proof of Thm.~\ref{thm:main})  it follows  there exists a function $f_0 \in \A_0(\mathcal{E}_c,M)$ such that
\begin{equation}
\left|I(f_0) - Q(f_0)\right|\geq M \left(\left(\frac{c^2-1+c}{c^2-1}\right)^{(c^2-1)/c}\left(\frac{c^2-1+c}{c}\right)\right)^{-2N}.
\end{equation}
Therefore, to have an error less than $\eps$ we need to take $N \geq N_l$, where
\begin{eqnarray}\label{eqn:lower-N}
   N_l= \frac{1}{2} \left(\ln \frac{M}{\eps}\right) \left(\ln\left( \left(\frac{c^2-1+c}{c^2-1}\right)^{(c^2-1)/c}\left(\frac{c^2-1+c}{c}\right)\right)\right)^{-1}.
\end{eqnarray}

Let us denote $\Delta=c-1$. Then for $c \to 1$ we obtain
\begin{eqnarray*}
  D & :=& \ln\left( \left(\frac{c^2-1+c}{c^2-1}\right)^{(c^2-1)/c}\left(\frac{c^2-1+c}{c}\right)\right)\\
  & = &  \frac{c^2-1}{c}\left( \ln(c^2-1+c)  - \ln(c-1)- \ln(c+1) \right) +  \ln\left(1+\frac{c^2-1}{c}\right)\\
  & = &  (2\Delta + O(\Delta^2))\left(\ln (1+O(\Delta)) - \ln \Delta - \ln (2+\Delta)\right) + \ln(1+O(\Delta)) \\
 & = & (2\Delta + O(\Delta^2))\left(O(\Delta) - \ln \Delta + O(1) \right) +  O(\Delta)\\
 & = & -2\Delta \ln \Delta + O(\Delta).
\end{eqnarray*}
Therefore
\begin{eqnarray*}
  D^{-1}&=& \frac{-1}{2 \Delta \ln \Delta} \left(1 + O\left(\left|\frac{1}{\ln \Delta}\right|\right)\right),
\end{eqnarray*}
and from~(\ref{eqn:lower-N}) we obtain
\begin{eqnarray*}
  N_l&=&\frac{-1}{4} \left(\ln \frac{M}{\eps}\right)  \frac{1}{ \Delta \ln \Delta} \left(1 + O\left(\left|\frac{1}{\ln \Delta}\right|\right)\right).
\end{eqnarray*}
\qed

\section{Conclusions}

For an ellipse $\mathcal{E}_c$ and $\alpha$ being the Lebesgue measure, let us compare $N_l$, the lower bound for the pieces of information required, with $N_g$, the estimate of the number of nodes in the Gauss-Legendre quadrature
needed  to obtain an error less than $\eps$, for $f \in \A_0(\mathcal{E}_c,M)$. From (\ref{eq:N-GL-nodes}) we obtain for $c \to 1^+$
\begin{eqnarray*}
  \frac{N_l}{N_g} & = & \frac{-\ln \frac{M}{\eps}}{4 (c-1) \ln(c-1)} \left(1 + O\left(\left|\frac{1}{\ln(c-1)}\right|\right)\right) \cdot \left(\frac{\ln \frac{M}{\eps} }{\ln c}\right)^{-1} \\
  & = & \frac{\ln c}{4 (c-1) \ln(c-1)}\left(1 + O\left(\left|\frac{1}{\ln(c-1)}\right|\right)\right)\\
  & = & \frac{1+ O(c-1)}{4\ln(c-1)}\left(1 + O\left(\left|\frac{1}{\ln(c-1)}\right|\right)\right)\\
  & \approx & \frac{1}{4 \ln (c-1)}.
\end{eqnarray*}
 We see that $ N_l/N_g \to 0$ for $c \to 1^+$, hence we have not obtained (\ref{eq:optimality-crit}). It will be interesting to see whether the lower bound can be improved to obtain a positive lower bound for this ratio not dependent on $c$. By Remark~\ref{rmk:chebyshev-2n} the estimate for error for the Gauss-Legendre
quadrature is optimal and the improvement should be sought through better estimation of $J_+(\mathcal{E}_c,N)$.


\begin{thebibliography}{Br}
 \bibitem[B67]{B} N. S. Bakhvalov, \emph{On the optimal speed of integrating analytic functions} (in Russian),
   Journal of Computational Mathematics and Mathematical Physics 7 (1967)
   63--75
[English translation: USSR Comput. Math. Math. Phys., 7:63-75.].

\bibitem[BC10]{Ban-Car 2008} {R.~Banuelos, T.~Carroll}, \textit{Stretching convex domains and the hyperbolic metric}, Quarterly J. Math. , 2010, 61 (3), 265-273.

\bibitem[Bo73]{Bo73} B. D. Bojanov, \emph{Optimal rate of integration and $\varepsilon$--entropy of a class of analytic functions} (in Russian), Mat. Zamietki,
   14 (1973), 3--10, English translation: Math. Notes, 19, 551--556


\bibitem[Bo74]{Bo74}  B. D. Bojanov, \emph{The best quadrature formula for a certain class of analytic functions}, Zastowania Matematyki, Applicationes Mathematicae, 14 (1974) 441--446


\bibitem[Br97]{Br} H. Brass, \emph{Quadraturverfahren}, Vandenhoeck and Ruprecht, G\"ottingen 1997

\bibitem[C78]{Con 1978} J. Conway, \textit{Functions of One Complex Variable I}, Graduate Texts in Mathematics,  Vol 11, Springer Verlag, 1978

\bibitem[C95]{Con 1995} J. Conway, \textit{Functions of One Complex Variable II}, Graduate Texts in Mathematics, Vol 159, Springer Verlag, 1995.

\bibitem[D70]{Dur 1970} {P.~ L.~ Duren}, \textit{Theory of $H^{p}$ spaces}, Pure and Applied Mathematics, Vol. 38 Academic Press, New York-London 1970.

\bibitem[G81]{Gar 1981} {J.~ B.~ Garnett}, \textit{Bounded analytic functions}, Pure and Applied Mathematics, 96. Academic Press, Inc. New York-London, 1981.


\bibitem[JP93]{Jar-Pfl 1993} {M.~Jarnicki, P.~Pflug}, \textit{Invariant Distances and Metrics in Complex Analysis}, De Gruyter Expositions in Mathematics 9, 1993.


\bibitem[K85]{KWW} {M. Kowalski, A. Wershulz H. Wo�niakowski}, \textit{Is Gauss quadrature optimal for analytic functions?} , Num. Math. 47 (1985) 89-98




\bibitem[O95]{O95} K. Yu.  Osipenko, \emph{Exact Values of n-Widths and Optimal Quadratures
on Classes of Bounded Analytic and Harmonic Functions}, Journal of
Approximation Theory 82 (1995) 156--175


\bibitem[P95]{P95} K. Petras, \emph{Gaussian  integration of
Chebyshev polynomials and analytic functions},  Numerical Algorithms 10(1995)  187--202

\bibitem[P98]{P98} K. Petras, \emph{Gaussian  Versus Optimal Integration of Analytic
Functions}, Constructive Approximation (1998) 14: 231--245



\bibitem[P02]{P02} K. Petras, \emph{Self-validating integration and approximation of piecewise analytic functions},
Journal of Computational and Applied Mathematics 145 (2002)
345--359






\bibitem[R69]{R} P. Rabinowitz, \emph{Rough and ready error estimates in Gaussian integration of
analytic  functions}, Comm. ACM, 12 (1969), 268--270


\bibitem[R66]{Rud 1966} {W. Rudin}, \textit{Real and Complex Analysis}, McGraw Hill Company, 1966

\bibitem[S63]{S}  I. F. Sharygin, \emph{Lower bounds for the errors of quadrature formulae in
classes of functions},  Journal of Computational Mathematics and
Mathematical Physics 3 (1963), 370--376



\bibitem[T08]{T} L. N. Trefethen, \emph{Is Gauss quadrature better than Clenshaw-Curtis?}, SIAM Review, 50 (2008), 67--87

\end{thebibliography}
\end{document}